\documentclass[a4paper,12pt]{article}
\usepackage{amssymb,amsthm,amsmath,amsfonts}
\usepackage{graphicx}
\usepackage{cite}
\usepackage{color}

 \newtheorem{thm}{Theorem}[section]
 \newtheorem{cor}[thm]{Corollary}
 
 \newtheorem{prop}[thm]{Proposition}
 \theoremstyle{definition}
 \newtheorem{defn}[thm]{Definition}
 \theoremstyle{remark}
 \newtheorem{rem}[thm]{Remark}
 
 \numberwithin{equation}{section}

\begin{document}

\title{On fiber bundles and  quaternionic slice regular functions}
\small{
\author
{Jos\'e Oscar Gonz\'alez-Cervantes}
\vskip 1truecm
\date{\small Departamento de Matem\'aticas, ESFM-Instituto Polit\'ecnico Nacional. 07338, Ciudad M\'exico, M\'exico\\ Email: jogc200678@gmail.com}}

\maketitle

\begin{abstract}

The papers \cite{O1,O2} are the first works to apply the theory of fiber bundles  in  the  study   of the quaternionic slice regular functions.

The main goal of the present work is to extend the results given
 in \cite{O1}, where the quaternionic right linear space of quaternionic slice regular functions was  presented as the  base space of a fiber bundle. When the quaternionic right linear space of quaternionic slice regular functions is associated to certain   domains then this paper shows that   the elements of total space, given  in \cite{O1},  are 
defined  from a pair of harmonic functions and a pair of orthogonal vectors. Simplifying the computations presented in \cite{O1},  where each   element  of the total space is formed by   two pair of conjugate harmonic functions and a pair of orthogonal unit vectors. 

This work  also gives some  interpretations of the behavior of the zero sets of some quaternionic slice regular polynomials 
 in terms of the theory of fiber bundles.

  \end{abstract}

\noindent
\textbf{Keywords.} Quaternionic slice regular functions,  Fiber bundles, Harmonic functions, Quaternionic slice regular polynomials,  Zero sets.\\
\textbf{AMS Subject Classification (2020):} Primary 30G35; Se\-con\-dary 46M20.

\section{Introducction}

The theory  of  fiber
 bundle in algebraic topology arises in 1950 and  N. E. Steenrod published the first textbook on fiber bundles in  1951, see  \cite{NS}. One the first applications of this theory was to give a mathematical interpretation of many physical phenomena, see \cite{BP, BD,  W}.   

In addition,    paper  \cite{O1} shows that  the quaternionic right linear space of slice regular functions is the total space of a fiber bundle intrinsically defined from the Representations Theorem and Splitting Lemma.  What is more,     the   quaternionic slice regular functions is defined on the total space of  some sphere bundles in paper \cite{O2}.

Let us recall that  given a two-dimensional harmonic function $u$  defined on a domain  one can find a holomorphic function such that Re$f=u$.  Particularly, if $D$ is a   disk then 
 the Schwarz's formula helps us  to obtain $f$. As a consequence of the above results this work shows a simplified version of the computations presented  in \cite{O1} and a supplement to   \cite{BW}, where the   relationship between the harmonicity with the slice regularity was studied. 

The zero sets of the  quaternionic slice regular polynomials
 was studied    in \cite{CSS2,GSt,V}.   Now this paper interprets   a relationship   between some   quaternionic slice regular polynomials with their  zero sets    in terms of the theory of fiber bundles. 

The structure of the paper is as follows:  Section \ref{pre} shows some basic facts about the 
conjugate harmonic functions, the zero sets of complex polynomials  and the theory   presented in \cite{O1}.   Section  \ref{MainResults} has three  subsections,  
the first one  presents  a simplified version of the fiber bundle  induced by the theory of quaternionic slice regular functions. 
	The relationship of some slice regular polynomials with its zero sets  
	is studied  using the  fiber bundles theory in Subsection \ref{ZeroSets}.

\section{Preliminaries}\label{pre}

Below we give basic definitions and facts on the  harmonic functions,  the zero set of complex polynomials  and on the fiber bundles. These notions will be used throughout the whole paper.

 \subsection{Basic definitions and facts on the  harmonic functions and complex polynomials}

Let us recall that  given a harmonic function  $u$ defined on  a 
simply connected domain $G\subset\mathbb C$  then a conjugate  harmonic of $u$ is  
\begin{align}\label{conj_Harm}v = \int_{(x_0,y_0)} ^{(x ,y)}  (-u_y dx + u_x dy),
\end{align}
and it is   uniquely determined up to an additive constant, see 
 \cite{NP} and if we assume that $G$ is the disk $|z|\leq \rho$ then  any holomorphic function with the real part $u$ is given by  
\begin{align}\label{SchwarzForm}\frac{1}{2\pi}  \int_0^{2\pi} u(\rho e^{it})\frac{\rho e^{it} + z }{\rho e^{it} - z}  dt + i \lambda,\end{align}
for $|z|< \rho$, where $\lambda\in\mathbb R$, see  the  Schwarz's Formula in \cite{A}.

On the other hand,  given a finite set $A\subset\mathbb C$ 
there exists an  unique monic complex polynomial $f$ such that $Z_f=A$, where $Z_f$ is the zero set of $f$. What is more, the  mapping   from the family of  finite subsets of $\mathbb C$  to the set of complex monic polynomials: $A\mapsto f$ is a    
bijective mapping. 
In addition, the Gauss-Lucas Theorem shows that $Z_{f'}\subset Kull (Z_f)$ for any complex polynomial $f$, where $Kull (Z_f)$ is the convex hull of $Z_f$ that is the intersection of all convex set that contain $Z_f$, see \cite{A,M}.

\subsection{Rudiments of quaternionic analysis and fiber bundles}

 The skew-field  of quaternions, denoted by  $\mathbb H$, consists of  $q=x_0  + x_{1} {e_1} +x_{2} e_2 + x_{3} e_3$ where $x_0, x_1, x_2, x_3\in\mathbb R$ and   
$e_1^2=e_2^2=e_3^2=-1$,  $e_1e_2=-e_2e_1=e_3$, $e_2e_3=-e_3e_2=e_1$, $e_3e_1=-e_1e_3=e_2$. The sets 
 $\{e_1,e_2,e_3\}$ and $\{1,e_1,e_2,e_3\}$ are     the standard basis of $\mathbb R^3$ and   $\mathbb H$, respectively. 
 The vector part of $q\in \mathbb H$ is  ${\bf{q}}= x_{1} {e_1} +x_{2} e_2 + x_{3} e_3$ and the  real part   is $q_0=x_0$. The conjugate quaternionic   of $q$  is   $\bar q=q_0-{\bf q} $ and its  norm  is  $\|q\|:=\sqrt{x_0^2 +x_1^2+x_2^3+x_3^2}= \sqrt{q\bar q} = \sqrt{\bar q  q}$.

The quaternionic unit open ball is  $\mathbb B^4(0,1):=\{q \in \mathbb H\ \mid \ \|q\|<1 \}$. 
Usually,  given $q\in \mathbb H$ and $\rho > 0$ denote 
$ \mathbb B^4(q, \rho )= \{r \in \mathbb H \ \mid  \   \|q-r\|< \rho \}$.
The unit  spheres   in  $\mathbb R^3$  and  in  $\mathbb H $ 
 are      
$   \mathbb{S}^2:=\{{\bf q}\in\mathbb R^3  \mid \|{\bf q}\| =1  \}$    and  $\mathbb{S}^3:=\{ {  q}\in\mathbb H \mid  \|{  q}\| =1   \} $,  respectively.

	The set  $ T$ consists of  $({\bf i},{\bf j}) \in \mathbb S^2 \times \mathbb S^2 $ such that  $ \{{\bf i},{\bf j},{\bf ij}\}$   is co-oriented with the standard basis of $ \mathbb R^3$.

Due to  ${\bf i}^{2}=-1$ for all ${\bf i}\in \mathbb S^2$  we see that    $\mathbb{C}({\bf i}):=\{x+{\bf i}y; \ |\ x,y\in\mathbb{R}\}\cong \mathbb C$ as fields. What is more, if  ${\bf q}\neq {\bf 0}$ then   $q$ can be rewritten by   $x+ {\bf I}_q y $ where  $x,y\in \mathbb R$ and ${\bf I}_q:=\|  {\bf q}\|^{-1}{\bf q}\in \mathbb S^2$. If    ${\bf q}={\bf 0}$ then choose $y=0$.

Given $u\in \mathbb S^3$, the mapping ${\bf q} \mapsto u{\bf q}\bar u$ for all ${\bf q}\in \mathbb R^3$ is a quaternionic rotation  that    preserves $\mathbb R^3$, see \cite{HJ}.  Define $R_{u}:T\to T$ by  $R_{u}({\bf i}, {\bf j}):= (u{\bf i}\bar u, u{\bf j}\bar u)$ for all $({\bf i}, {\bf j})\in T$ and consider the norm $\|\cdot\|_{\mathbb R^6}$ in  $T$.

   A real differentiable function $f:\Omega\to \mathbb{H}$, where $\Omega\subset\mathbb H$ is an open,
		is called left slice regular function, or slice regular function on $\Omega$,   
if   
\begin{align*}
\overline{\partial}_{{\bf i}}f\mid_{_{\Omega\cap \mathbb C({\bf i})}}:=\frac{1}{2}\left (\frac{\partial}{\partial x}+{\bf i} \frac{\partial}{\partial y}\right )f\mid_{_{\Omega\cap \mathbb C({\bf i})}}=0  \textrm{ on  $\Omega\cap \mathbb C({\bf i})$,}
\end{align*}
for all ${\bf i}\in \mathbb{S}^2$. The  derivative of $f$, or Cullen's  derivative,  is   
$f'=\displaystyle 
 {\partial}_{{\bf i}}f\mid_{_{\Omega\cap \mathbb C({\bf i})}} = \frac{\partial}{\partial x} f\mid_{_{\Omega\cap \mathbb C({\bf i})}}= \partial_xf\mid_{_{\Omega\cap \mathbb C({\bf i})}}$.
The quaternionic right linear space of the  slice regular functions on $\Omega$ is denoted by $\mathcal{SR}(\Omega)$, see  \cite{CGS3, newadvances, CSS, CSS2, GenSS}.  A set $U\subset\mathbb H$  is called axially symmetric  slice domain, or  axially symmetric  s-domain, if $U\cap \mathbb R\neq \emptyset$, $U_{\bf i} = U\cap \mathbb C({\bf i})$ is  a domain in   $\mathbb C({\bf i})$  for all ${\bf i}\in\mathbb S^2$, and 
 $x+{\bf i}y \in U$ with $x,y\in\mathbb R$  implies  $\{x+{\bf j}y \ \mid  \ {\bf j}\in\mathbb{S}^2\}\subset U$.  

Given ${\bf i }\in\mathbb S^2$ an  axially symmetric s-domain $\Omega\subset \mathbb H$ 
shall be called ${\bf i}$-simply connected  axially symmetric s-domain
 if  
$\Omega_{\bf i}\subset\mathbb C({\bf i}) $ 
is domain simply connected. From  the symmetric property of $\Omega$  one  sees that  
$\Omega_{\bf j}\subset \mathbb C({\bf j})$ is a simply connected domain  for all  ${\bf j}\in \mathbb S^2$.
The axially symmetric s-domains $\Omega, \Xi\subset \mathbb H$  are      ${\bf i}$-conformally equivalent iff    there exists   $\alpha\in \mathcal {SR}(\Omega)$ such that  $Q_{\bf {i,j}}[\alpha] : \Omega_{\bf i}\to \Xi_{\bf i}  $ is a biholomorphism.

Let us mention an important two property of the quaternionic slice regular functions.

  Splitting Lemma. Given an  axially symmetric  s-domain $\Omega\subset \mathbb H$ and $f\in \mathcal {SR}(\Omega)$.  For every ${\bf i},{\bf j}\in \mathbb{S}$, orthogonal  to  each other,
there exist $F,G\in \textit{Hol}(\Omega_{\bf i})$,   holomorphic functions,  
such that  $f_{\mid_{\Omega_{\bf i}}} =F +G  {\bf j}$ on $\Omega_{\bf i}$, see \cite{CSS}.

Representation Formula. 
Given an  axially symmetric  s-domain $\Omega\subset \mathbb H$ and $f\in \mathcal {SR}(\Omega)$.
 For every  $q=x+{\bf I}_q y \in \Omega$ with $x,y\in\mathbb R$ and  ${\bf I}_q \in \mathbb S^2$ one has that 
\begin{align*}
f(x+{\bf I}_q y) = \frac {1}{2}[   f(x+{\bf i}y)+ f(x-{\bf i}y)]
+ \frac {1}{2} {\bf I}_q {\bf i}[ f(x-{\bf i}y)- f(x+{\bf i}y)],
\end{align*} 
for all ${\bf i}\in \mathbb S^2$, see  \cite{newadvances}.

From the previous results one has   the  following  operators:
\begin{itemize}
\item $
     Q_{ {\bf i},{\bf j}} :    \mathcal{SR}(\Omega) \to \textrm{Hol}(\Omega_{  {\bf i}})+ \textrm{Hol}(\Omega_{ {\bf i}}){  {\bf j} }$ given by 
$ Q_{ {\bf i} , {\bf j} } [f ]= f\mid_{\Omega_{{\bf i}}} = f_1+f_2{\bf j}$ for all $f\in\mathcal{SR}(\Omega)$ where $f_1,f_2\in \textrm{Hol}(\Omega_{\bf i})$, space of holomorphic function on $\Omega_{\bf i}$. 
\item $ P_{  {\bf i},{\bf j} } :  \textrm{Hol}(\Omega_{  {\bf i}})+ \textrm{Hol}(\Omega_{  {\bf i}}){  {\bf j}} \to  \mathcal{SR}(\Omega)$  defined by 
\begin{align}\label{operatorsP}
   P_{ {\bf i},{\bf j} }[g](q)= \frac{1}{2}\left[(1+ {\bf I}_q{ \bf  i})g(x-y{{\bf i}}) + (1- {\bf I}_q {  {\bf i}}) g(x+y {  {\bf i} })\right]
	, 	\end{align} 
	for all  $g\in  \textrm{Hol}(\Omega_{  {\bf i} })+ \textrm{Hol}(\Omega_{  {\bf i} }){  {\bf j} }$ where  $ q=x+{\bf I}_qy \in \Omega$.
	\end{itemize}
	  What is more,
 \begin{align}\label{ecua111} P_{  {\bf i} ,  {\bf j}  }\circ Q_{ {\bf i} ,{\bf j}  }= \mathcal I_{\mathcal{SR}(\Omega)} \quad \textrm{and} \quad    Q_{  {\bf i} ,  {\bf j} }\circ P_{ {\bf i} ,  {\bf j}   }= \mathcal I_{ \textrm{Hol}(\Omega_{ {\bf i} })+ \textrm{Hol}(\Omega_{ {\bf i} }){ {\bf j} } },
\end{align} 
where $\mathcal I_{\mathcal{SR}(\Omega)}$  and $\mathcal I_{ \textrm{Hol}(\Omega_{ {\bf i} })+ \textrm{Hol}(\Omega_{ {\bf i} }){  {\bf j}  } }$ are  the identity operators in  $\mathcal{SR}(\Omega)$ and    in  ${\textrm{Hol}(\Omega_{ {\bf i} })+ \textrm{Hol}(\Omega_{ {\bf i} }){ {\bf j} } }$, respectively, see  \cite{GS}.

Given  $g\in \mathcal{SR}(\Omega)$ and     $({\bf i},{\bf j}) \in T$,  the real components of $Q_{{\bf i},{\bf j}}[g]$ are given by
\begin{align}\label{equa1234} &  D_1[g, {\bf i}, {\bf j}   ] :=   \frac{    Q_{{\bf i},{\bf j}}[g] + \overline{   Q_{{\bf i},{\bf j}}[g]  }   }{ 2 }  ,  & \  D_2[g, {\bf i}, {\bf j}     ] :=  - \frac{  Q_{{\bf i},{\bf j} }[g] {\bf i}  +  \overline{   Q_{{\bf i} , {\bf j} }[g]   {\bf i}  }   }{2}   , \nonumber \\
&  D_3[g, {\bf i}, {\bf j}      ] :=   - \frac{ Q_{{\bf i},{\bf j} }[g]  {\bf j}   + \overline{   Q_{{\bf i},{\bf j} }[g]   {\bf j} }   }{2}   , &  \  D_4[g, {\bf i}, {\bf j}     ] :=    \frac{ Q_{ {\bf i}, {\bf j} }[g] {\bf j}{\bf i}  + \overline{   Q_{ {\bf i},{\bf j} }[g]   {\bf j}{\bf i}    }   }{2} ,
\end{align}
i.e.,
\begin{align}\label{ecua678} Q_{{\bf i},{\bf j}}[g] =  D_1[g, {\bf i}, {\bf j}   ]  +  D_2[g, {\bf i}, {\bf j}   ] {\bf i}   +  D_3[g, {\bf i}, {\bf j}   ] {\bf j}+  D_4[g, {\bf i}, {\bf j}  ] {\bf i}{\bf j} .\end{align}

 Let us recall several properties of the quaternionic slice regular functions:  For any    $f,g \in \mathcal {SR}(\Omega) $  define   
\begin{align}\label{iproducto} 
f\bullet_{_{{\bf i}, {\bf j}}} g := P_{\bf{i},{\bf j}} [ \  f_1g_1+ f_2g_2 {\bf j} \ ] ,
\end{align}
   where $Q_{{\bf i}, {\bf j}}[f]=f_1+f_2{\bf j}$ and  $Q_{{\bf i}, {\bf j}}[g]=g_1+g_2{\bf j}$ with $f_1,f_2,g_1,g_2\in \textrm{Hol}(\Omega_{\bf i})$. Also note that  
 \begin{align}\label{identity1} f + {\bf i }f {\bf i } = 2 f_1   \quad \textrm{and} \quad     
f - {\bf i }f {\bf i } =  2 f_2 {\bf j }
,
\end{align}
on $\Omega_{\bf i}$. What is more,  if $f,g\in \mathcal {SR} (\mathbb B^4(0,1))$  there exist $(a_n) , (b_n)\subset \mathbb H$ such that  
$$ 
f(q)= \sum_{n=0}^{\infty} q^n a_n, \quad   g(q)= \sum_{n=0}^{\infty} q^n b_n \quad \textrm{ and } \quad  
 f*g(q):= \sum_{n=0}^{\infty} q^n  \sum_{k=0}^{n} a_k b_{n-k},$$ 
 for all $q\in\mathbb B^4(0,1)$, see \cite{newadvances}
 By $\mathcal {PSR(\mathbb H)} $ we mean the set of  the monic quaternionic slice regular polynomials, i.e.,
 $f\in  \mathcal {PSR(\mathbb H)} $ iff there exists $a_0,a_1 \dots, a_{n-1}\in \mathbb H$ such that 
$$f(q)= a_0+ qa_1+ \cdots + q^{n-1} a_{n-1} + q^n.$$
Denote      
    $\mathcal{SR}_c(\Omega):= \mathcal{SR}(\Omega)\cap C(\overline{\Omega},\mathbb H)$.

On the other hand,   in \cite{Bredon} it is found  that  
 a  fiber bundle is denoted as $(X, {\bf{P}}, B, F)$, 
where $X$,  $B$ and $F$ are   Hausdorff spaces and are called  total space,   base space and  the fiber space, respectively.   There exists a   topological group $K$, called   structure group,  acting on $F$ as a group of homeomorphisms.  The  bundle projection ${\bf{P}}: X\to B$ satisfies that  for each   element of $ B$ has  a neighborhood $U\subset B$    and  
  a homeomorphism $\varphi : U\times F\to {\bf{P}}^{-1}(U)$, called a   trivialization  over $U$,  such that 
	${\bf{P}}\circ \varphi (b, y) = b$ for all $b\in U$ and   $y\in F$. 
The family of  all   triviali\-za\-tions,   $\Phi$,  satisfy that 
\begin{enumerate}
\item   If $\varphi:U\times F\to {\bf P}^{-1}(U)$ belongs to  $\Phi$ and if $V\subset U$ then     $\varphi_{\mid_{V\times F}}$ belongs to $\Phi$.
\item If $\varphi,  \phi\in\Phi$ are trivializations  over $U\subset B$  then there exists a map $\psi :U\to K$ such that 
$ \phi (u,y) = \varphi ( u, \psi(u) (y) )$ 
for all $u\in U$ and $y\in F$.      
 \item  $\Phi$ is a maximal family that  satisfies  the previous facts.
\end{enumerate}
A continuous map  $S: B\to X$  such that ${\bf P}\circ S(x)=x$ for all $x\in B$ is a  
   section of $(X, {\bf{P}}, B, F)$, see  \cite{NS} and  an
		arbitrary map $\mathfrak K:A\to B$, where   $A$ is  a nonempty set, induce the  pullback  fiber bundle
			$(\mathfrak K^*(X), {\bf P}', A, F)$,  where 
 $\mathfrak K^*(X) := \{  (a,x)\in A\times X \ \mid \ {\bf P}(x)= \mathfrak K(a)     \} $   and    $ {\bf P}'(a,x)= a$ for all  $(a,x)\in \mathfrak K^*(X)$,   see  \cite{Bredon, AH}.
	
  A  morphism between  two fiber bundles  $$\Gamma : (X_1, {\bf P}_1, B_1, F_1) \to   (X_2, {\bf P}_2, B_2, F_2)$$    
	consists of a pair of  continuous maps  $\Gamma_1: X_1 \to X_2$ and $  \Gamma_2 : B_1 \to  B_2$  such that  the  
diagram 
\begin{align*} \begin{array}{ rcl}                &     {\bf P}_2     &                     \\
                             X_2 & \longrightarrow & B_2  \\
							\Gamma_1 \ \  \uparrow	&							&  \uparrow	 \  \    \Gamma_2 \\		
 														X_1      & \longrightarrow & B_1             				\\
						   					             &     {\bf P}_1     &                      \\
													\end{array}  \end{align*}
				commutes and there exists   a   morphism  $\Gamma^{-1} : (X_2, {\bf P}_2, B_2, F_2) \to   (X_1, {\bf P}_1, B_1, F_1)$  				such  that  
$\Gamma\circ \Gamma^{-1}$ and $\Gamma^{-1}\circ \Gamma$ are the identity morphisms, then $\Gamma$ is an isomorphism,   see  \cite{Bredon, AH}.

An interesting application of the previous theory in the quaternionic slice regular functions was  presented in    \cite{O1}  showing that  
     $(\mathcal H(S_{\Omega}),  \mathcal{P}_{\Omega} ,\mathcal{SR}_c(\Omega), T  )$ is a kind of bi-sphere
bundle, where  $\Omega\subset\mathbb H$ is a bounded   axially symmetric  s-domain, the   set of   pairs of conjugated harmonic functions    on 
\begin{align*}
  \displaystyle  S_\Omega:=\{(x,y) \in \mathbb R ^2 \ \mid  \textrm{ there exists } {\bf i} \in \mathbb S^2  \textrm{ such that } x+y{\bf i} \in \Omega\}
\end{align*}
    that are  conti\-nuous on  $\overline{S_{\Omega}}$ is denoted by 
   $Harm_c^2(S_\Omega)$,  
\begin{align*} 
&	\mathcal H( {S_\Omega})   : = \left\{  \left( \left( \begin{array}{cc}   a  &  b   \\  c & d  \end{array}\right)  ,   ({\bf i},{\bf j})  \right) \   \mid   \    (a,b), (c,d) \in   Harm_c^2( {S_\Omega})  , \  ({\bf i},{\bf j}) \in T \right\}.
\end{align*}    
The projection  bundle 	${ \mathcal{P}}_{\Omega} : \mathcal H( {S_\Omega})   \to \mathcal{SR}_c(\Omega)$ is 
\begin{align*} { \mathcal{P}}_{\Omega}   \left( \left( \begin{array}{cc}   a  &  b   \\  c & d  \end{array}\right)  ,  ({\bf i},{\bf j})  \right)    := P_{{\bf i},{\bf j}}[  a +b {\bf i} +c{\bf j}  + d {\bf i}{\bf j}       ],\end{align*}
and the structure group is $K=\{R_u  \mid u\in \mathbb S^3 \}$. 
 The trivializations
 are given by  
\begin{align*}  \varphi_u  [ f, ({\bf i},{\bf j}) ] := &  \left( \left( \begin{array}{cc}  \displaystyle   D_1[f, u{\bf i}\bar u, u{\bf j}\bar u  ]  &  D_2[f, u{\bf i}\bar u, u{\bf j}\bar u  ] \\   D_3[f, u{\bf i}\bar u, u{\bf j}\bar u  ]  &  D_4[f, u{\bf i}\bar u, u{\bf j}\bar u  ]    \end{array}\right)  ,   (u{\bf i}\bar u, u{\bf j}\bar u )  \right)  ,     \end{align*}  
for all $(f, ({\bf i}, {\bf j}))\in U\times T$, where    $U\subset\mathcal{SR}_c(\Omega)$ is  a neighborhood and   $u\in \mathbb S^3$.

\section{Main results}\label{MainResults}

\subsection{A fiber bundle over the quaternionic slice regular functions}\label{subseccion1}

Paper  \cite{BW} presents several relationships between   the harmonic function theory with  the slice regular functions.
 This subsection  presents a   consequence of the harmonicity in  the interpretation of the    quaternionic slice regular functions 
 in terms of  fiber bundles expanding the results obtained in   \cite{O1}.

\begin{defn} \label{DefFiberBundle} Given ${\bf i}\in\mathbb S^2$ and let $\Omega\subset \mathbb H$ be a 
$ {\bf i}$-simply connected     bounded axially symmetric s-domain.
\begin{enumerate}
\item The set $C\textrm{Harm}(\Omega)$ is formed  by harmonic functions $a$ on $\Omega$ such that  $a$,  $a_x$ and $a_y$ are    continuous  on $\bar\Omega$. Given $a,b\in C\textrm{Harm} (\Omega)$ 
define $a\sim b$ iff there exits   a real constant $ \lambda$ such that   $a=b + \lambda  $.  Define 
$$N_{C\textrm{Harm}(\Omega)}( a ):=\|a\|_{\infty} + \|a_x\|_{\infty}+ \|a_y\|_{\infty}, $$ for all $a\in C\textrm{Harm}(\Omega)$ and in  
		  $C\textrm{Harm} (\Omega) \diagup \sim$ consider the norm 
			$$N_1 ([ a] ):=\inf\{  N_{C\textrm{Harm}(\Omega)}( b ) \  \mid \ b\in [a] \} , $$ for all $[a]\in C\textrm{Harm} (\Omega) \diagup \sim$
			and define  
$${\bf X}_{\Omega}:= (C\textrm{Harm} (\Omega) \diagup \sim)\times (C\textrm{Harm} (\Omega) \diagup \sim) \times T,$$ equipped with the norm 
$$N_{{\bf X}_{\Omega}} ([ a], [c], ({\bf i}, {\bf j}) ) = N_1 ([a]) +N_1 ([c]) + \| ({\bf i}, {\bf j}) \|_{\mathbb R^6}. $$

\item The function set  $  C\mathcal{SR}(\Omega)$ consists of $f \in \mathcal{SR}(\Omega)$ such that   $f$ and $f'$ 	are	  continuous on $\bar \Omega$.  Given $ f,g\in  C\mathcal{SR}(\Omega)$ define $f\sim g$ iff there exists $c\in \mathbb H$ such that $f=g + {c} $. 

Consider ${\bf B}_{\Omega}=  C\mathcal{SR}(\Omega) \diagup \sim$ equipped with the norm 
$$N_{{\bf B}_{\Omega}} ( [ f] ):=\inf\{  \| g\|_{\infty}  \  \mid \ g\in [f] \} , \quad \forall [f]\in {\bf B}_{\Omega}.$$

\item Given $\left(   [a] , [c]    ,  ({\bf i},{\bf j})  \right) \in {\bf X}{\Omega}$ define  	
\begin{align*} {  {\bf P}}_{\Omega}  \left(   [a] , [c]    ,  ({\bf i},{\bf j})  \right)    :=
		[\ P_{{\bf i},{\bf j}}(   \ a  +c{\bf j}  + {\bf i}\int_{(x_0,y_0)} ^{(x ,y)}  -  ( a   + c  {\bf j})_y  dx +   (a  + c   {\bf j})_x   dy     \ )  \ ].\end{align*}
		where the integration is on any path contained  in $S_\Omega$  from $(x_0,y_0)$ to $(x,y)$.

\item  Given     $({\bf i}, {\bf j})\in T $ set    
$$\displaystyle 
{\bf S}_{\Omega, {\bf i}, {\bf j}}([f]):= \left( \ [   D_1(f,  {\bf i} ,  {\bf j})  ]  ,  [  D_3(f,  {\bf i} ,  {\bf j} ) ]   ,
(  {\bf i} , {\bf j} ) \right)  ,       
$$   
for  all $ [f] \in {\bf B} $.

\item    Let  $U\subset {\bf B}$ be a neighborhood and   $u\in \mathbb S^3$. Denote  
\begin{align*}  \varphi_{u}  ( [f], ({\bf i},{\bf j}) )  :=  \left(  [ \displaystyle   D_1(f, u{\bf i}\bar u, u{\bf j}\bar u ) ] ,  [D_3(f, u{\bf i}\bar u, u{\bf j}\bar u ) ]   ,   (u{\bf i}\bar u, u{\bf j}\bar u )  \right)  ,     \end{align*}  
for all $([f], ({\bf i}, {\bf j}))\in U\times T$.

\end{enumerate}
\end{defn}

\begin{rem}

Note the equivalence relations   established to obtain 
 ${\bf X}_{\Omega}$  and ${\bf B}_{\Omega}$ allow us  to see    that      ${  {\bf P}}_{\Omega}$ is   well-defined and to  use $T$ as our fiber space.
\end{rem}

\begin{prop} \label{FiberBundle}     $({\bf X}_{\Omega},  {\bf P}_{\Omega}  ,{\bf B}_{\Omega}, T  )$  is a fiber bundle where the structure group is $K=\{R_u  \mid u\in \mathbb S^3 \}$ and  the  family of trivializations      is  $\Phi =\{ \varphi_u    \mid  u\in\mathbb S^3\}$. 
\end{prop}
\begin{proof} The normed space   ${\bf X}_{\Omega}$,  ${\bf B}_{\Omega}$ and $ T$ are Hausdorff  spaces. 
Given $ a,c, p,r \in C\textrm{Harm}(\Omega)$ and     $({\bf i},{\bf j}),
({\bf k},{\bf l})  \in T$ from  Proposition 3.4 presented in \cite{O1} one obtains  
\begin{align*} & \| P_{{\bf i}, {\bf j}}\left(  a +b {\bf i} +c{\bf j}  + d {\bf i}{\bf j}       \right) -   P_{{\bf k},{\bf l}}\left(  p +  q {\bf k} + r {\bf l}  + s {\bf kl}      \right) \|_{\infty}  \\
 \leq &  2 (  \   \|a-p\|_{\infty}  + \|b-q\|_{\infty} + \|c-r\|_{\infty} +\|d-s\|_{\infty} + \|({\bf i}, {\bf j})-({\bf k},{\bf l}) \|_{\mathbb R^6} \ )  \\
 &  ( \    \|p\|_{\infty}  + \|q\|_{\infty} + \|r\|_{\infty} +\|s\|_{\infty}  +1  \ ),
\end{align*}
where the harmonic functions $b$, $d$, $q$ and $r$ are given in terms of $a$, $c$, $p$ and  $r$  according  to   \eqref{conj_Harm},  respectively. As $\Omega\subset \mathbb H$ is a bounded set   there exist two constants $k_1,k_2 >0$ such that 
\begin{align*}\|b-q\|_{\infty}  \leq    & k_1 \left( \|a_x-p_x\|_{\infty} + \|a_y-p_y\|_{\infty} \right)     \\
\|d-s\|_{\infty}  \leq    &  k_1\left(  \|c_x-r_x\|_{\infty} + \|c_y-r_y\|_{\infty} \right)
\end{align*}
and   
\begin{align*} & \| P_{{\bf i}, {\bf j}}\left(  a +b {\bf i} +c{\bf j}  + d {\bf i}{\bf j}       \right) -   P_{{\bf k},{\bf l}}\left(  p +  q {\bf k} + r {\bf l}  + s {\bf kl}      \right) \|_{\infty}  \\
 \leq &  k_2 (  \   \|a-p\|_{\infty}  +    \|a_x-p_x\|_{\infty} +  \|a_y-p_y\|_{\infty}
 + \|c-r\|_{\infty} +\|c_x-r_x\|_{\infty}   \\
 & \vspace{.5cm} +\|c_y-r_y\|_{\infty} + \|({\bf i}, {\bf j})-({\bf k},{\bf l}) \|_{\mathbb R^6} \ )   ( \    \|p\|_{\infty}  + \|p_x\|_{\infty} + \|p_y\|_{\infty}   + \|r\|_{\infty}  \\
 & \vspace{.5cm} + \|r_x\|_{\infty}+ \|r_y\|_{\infty}  +1  \ ) 
= \left(  \  N_{C\textrm{Harm}(\Omega)}( a-p ) + N_{C\textrm{Harm}(\Omega)}( c-r ) \right.  \\ 
 &  \left. + \|({\bf i}, {\bf j})-({\bf k},{\bf l}) \|_{\mathbb R^6} \ \right)  \left( N_{C\textrm{Harm}(\Omega)}( p ) + N_{C\textrm{Harm}(\Omega)}( r ) + 1 \right).
\end{align*}
By the  properties of the infimum of a set one obtains that 
\begin{align*} & N_{{\bf B}_{\Omega}} ( {  {\bf P}}_{\Omega}  \left(   [a] , [c]    ,  ({\bf i},{\bf j})  \right)  -  {  {\bf P}}_{\Omega}  \left(   [p] , [r]    ,  ({\bf k},{\bf l})  \right)    )  \\
\leq & k_2 \left[  N_{{\bf X}_{\Omega}} (  \  [ a] -[p], [c]-[r] , ({\bf i}-{\bf k} , {\bf j}    -   {\bf l})  \   )      \right]   (  N_1 ([p]) +N_1 ([r]) + 1 ) .
\end{align*}
Therefore,     $ {\bf P}_{\Omega} : {\bf X}_{\Omega} \to {\bf B}_{\Omega}$ 
is a continuous mapping.

Analogously to the computations presented in  Proposition 3.6 given in \cite{O1} one obtains that  
 $K=\{R_u \ \mid \  u\in \mathbb S^3\}$ equipped with the composition is a topological group which acts on $T$ as a group of homeomorphisms.

Let $U\subset {\bf B}$
 be  a neighborhood. The operator $\varphi_u : U\times  T \to  {\bf P
}^{-1}(U)  $  is a homeomorphism. 

1.-   $\varphi_u $ is a bijective operator.  If  $  \varphi_u  ( [f],({\bf i},{\bf j})) =  \varphi_u  ([ g],({\bf k},{\bf l})]) $ 
then $({\bf i},{\bf j})=({\bf k},{\bf l})$ and there exists $\lambda_1,\lambda_2 \in \mathbb R$ such that  
\begin{align*}
    D_1(f, u{\bf i}\bar u, u{\bf j}\bar u ) =  & D_1(g, u{\bf i}\bar u, u{\bf j}\bar u )+ \lambda_1 ,\\
    D_3(f, u{\bf i}\bar u, u{\bf j}\bar u ) = &  D_3(f, u{\bf i}\bar u, u{\bf j}\bar u )  + \lambda_2
\end{align*}
and using  \eqref{conj_Harm} one concludes that $Q_{u{\bf i}\bar u, u{\bf j}\bar  u }[f] = Q_
{u{\bf i}\bar u, u{\bf j}\bar  u }[g] + c $, where $c\in \mathbb H$ and  applying  $P_{u{\bf i}\bar u, u{\bf j}\bar  u }$ one obtains that $[f]=[g]$.
\item Given $
 \displaystyle \left( [ a],  [c ],   ({\bf i},{\bf j})  \right) \in {\bf P
}^{-1}(U) $. 
Denote   $({\bf k} ,{\bf l})   =( \bar  u {\bf i}  u,  \bar  u {\bf j}  u) \in T$ and   
$$f = P_{{u{\bf k}\bar u},{u {\bf l}\bar u }}(   \ (a  +c{u {\bf l}\bar u })+ {u{\bf k}\bar u}\int_{(x_0,y_0)} ^{(x ,y)}  -  ( a   + c  {u {\bf l}\bar u })_y  dx +   (a  + c   {u {\bf l}\bar u })_x   dy     \ )   ,$$
 that satisfy $ \displaystyle \varphi_u ([f],({\bf k}, {\bf l})) =  \left( [ a]  , [ c]   ,   ({\bf i},  {\bf j})  \right)$.

2.- The inverse operator of $\varphi_u $ is given by 
\begin{align*}\varphi_u^{-1}  \left( [  a] , [ c]  ,   ({\bf k} ,   {\bf l} )  \right) = ( \ {\bf P}_{\Omega} ( [  a], [ c],  
  (\bar u{\bf k}u ,  \bar u {\bf l}u )  )   , (\bar u{\bf k}u,  \bar  u{\bf l} u  ) \ ) .
	\end{align*}

3.- Continuity of $\varphi_u $. Given   $[f],[g]\in U$ and $({\bf i},{\bf j}), ({\bf k},{\bf l})\in \mathbb S^2$, from Proposition 3.5 presented in \cite{O1} one has that
\begin{align*}   &  \|  \displaystyle   D_1(f, u{\bf i}\bar u, u{\bf j}\bar u ) (x,y)-  
 \displaystyle   D_1(g, u{\bf k}\bar u, u{\bf l}\bar u ) (x,y) \| \\
&  + \|  \displaystyle   D_3(f, u{\bf i}\bar u, u{\bf j}\bar u ) (x,y) -   \displaystyle   D_3(g, u{\bf k}\bar u, u{\bf l}\bar u ) (x,y) \| \\ 
 &  + \| (u{\bf i}\bar u    , u{\bf j}\bar u ) - (u{\bf k}\bar u, u{\bf l}\bar u) \|_{\mathbb R^6}    \\
   \leq  &  4 \| f-g \|_{\infty} +   4 \| g(x+{\bf i}y) - g(x+{\bf k}y) \|    \\   
																	&   +   (2 \|   g \|_{\infty}   +1) \|({\bf i}, {\bf j}) - ({\bf k},{\bf l})\|_{\mathbb R^6}. 
\end{align*}
Repeating the previous computations for $f'$,  $g'$ and recalling that this deri\-va\-tive on each slice  $\Omega_{\bf i}$ is given by $\dfrac{\partial}{\partial x}$, or equivalently by $-{\bf i}\dfrac{\partial}{\partial y}$, one concludes that 
\begin{align*}   &  \|  \displaystyle   D_1(f, u{\bf i}\bar u, u{\bf j}\bar u )_x (x,y)-  
 \displaystyle   D_1(g, u{\bf k}\bar u, u{\bf l}\bar u )_x (x,y) \| \\
&  + \|  \displaystyle   D_3(f, u{\bf i}\bar u, u{\bf j}\bar u )_x (x,y) -   \displaystyle   D_3(g, u{\bf k}\bar u, u{\bf l}\bar u )_x (x,y) \| \\ 
 &  + \| (u{\bf i}\bar u    , u{\bf j}\bar u ) - (u{\bf k}\bar u, u{\bf l}\bar u) \|_{\mathbb R^6}    \\
   \leq  &  4 \| f'-g' \|_{\infty} +   4 \| g'(x+{\bf i}y) - g'(x+{\bf k}y) \|    \\   
																	&   +   (2 \|   g' \|_{\infty}   +1) \|({\bf i}, {\bf j}) - ({\bf k},{\bf l})\|_{\mathbb R^6}. 
\end{align*}
and 
\begin{align*}   &  \|  \displaystyle   D_1(f, u{\bf i}\bar u, u{\bf j}\bar u )_y (x,y)-  
 \displaystyle   D_1(g, u{\bf k}\bar u, u{\bf l}\bar u )_y (x,y) \| \\
&  + \|  \displaystyle   D_3(f, u{\bf i}\bar u, u{\bf j}\bar u )_y (x,y) -   \displaystyle   D_3(g, u{\bf k}\bar u, u{\bf l}\bar u )_y (x,y) \| \\ 
 &  + \| (u{\bf i}\bar u    , u{\bf j}\bar u ) - (u{\bf k}\bar u, u{\bf l}\bar u) \|_{\mathbb R^6}    \\
   \leq  &  4 \| f'-g' \|_{\infty} +   4 \| g'(x+{\bf i}y) - g'(x+{\bf k}y) \|    \\   
																	&   +   (2 \|   g' \|_{\infty}   +1) \|({\bf i}, {\bf j}) - ({\bf k},{\bf l})\|_{\mathbb R^6}. 
\end{align*}
Therefore 

\begin{align*}   &  \|  \displaystyle   D_1(f, u{\bf i}\bar u, u{\bf j}\bar u ) (x,y)-  
 \displaystyle   D_1(g, u{\bf k}\bar u, u{\bf l}\bar u ) (x,y) \| \\
&  + \|  \displaystyle   D_3(f, u{\bf i}\bar u, u{\bf j}\bar u ) (x,y) -   \displaystyle   D_3(g, u{\bf k}\bar u, u{\bf l}\bar u ) (x,y) \| \\ 
 &  + \| (u{\bf i}\bar u    , u{\bf j}\bar u ) - (u{\bf k}\bar u, u{\bf l}\bar u) \|_{\mathbb R^6}    \\
  & +  \|  \displaystyle   D_1(f, u{\bf i}\bar u, u{\bf j}\bar u )_x (x,y)-  
 \displaystyle   D_1(g, u{\bf k}\bar u, u{\bf l}\bar u )_x (x,y) \| \\
&  + \|  \displaystyle   D_3(f, u{\bf i}\bar u, u{\bf j}\bar u )_x (x,y) -   \displaystyle   D_3(g, u{\bf k}\bar u, u{\bf l}\bar u )_x (x,y) \| \\ 
 &  + \| (u{\bf i}\bar u    , u{\bf j}\bar u ) - (u{\bf k}\bar u, u{\bf l}\bar u) \|_{\mathbb R^6}    \\
& + \|  \displaystyle   D_1(f, u{\bf i}\bar u, u{\bf j}\bar u )_y (x,y)-  
 \displaystyle   D_1(g, u{\bf k}\bar u, u{\bf l}\bar u )_y (x,y) \| \\
&  + \|  \displaystyle   D_3(f, u{\bf i}\bar u, u{\bf j}\bar u )_y (x,y) -   \displaystyle   D_3(g, u{\bf k}\bar u, u{\bf l}\bar u )_y (x,y) \| \\ 
 &  + \| (u{\bf i}\bar u    , u{\bf j}\bar u ) - (u{\bf k}\bar u, u{\bf l}\bar u) \|_{\mathbb R^6}    \\
   \leq  &  4 \| f-g \|_{\infty} +   4 \| g(x+{\bf i}y) - g(x+{\bf k}y) \|    \\   
																	&   +   (2 \|   g \|_{\infty}   +1) \|({\bf i}, {\bf j}) - ({\bf k},{\bf l})\|_{\mathbb R^6} \\
																	&  +   4 \| f'-g' \|_{\infty} +   4 \| g'(x+{\bf i}y) - g'(x+{\bf k}y) \|    \\   
																	&   +   (2 \|   g' \|_{\infty}   +1) \|({\bf i}, {\bf j}) - ({\bf k},{\bf l})\|_{\mathbb R^6} \\
     & +  4 \| f'-g' \|_{\infty} +   4 \| g'(x+{\bf i}y) - g'(x+{\bf k}y) \|    \\   
																	&   +   (2 \|   g' \|_{\infty}   +1) \|({\bf i}, {\bf j}) - ({\bf k},{\bf l})\|_{\mathbb R^6}. 
\end{align*}
If there is a sequence $(f_n) \in  C\mathcal{SR}(\Omega) $ and $g\in  C\mathcal{SR}(\Omega) $ such that $\|f_n-g\|_{\infty}\to 0$. Then from the Splitting Lemma   on each slice we have two sequences of holomorphic functions that  converge uniformly to holomorphic functions in the compact set $\bar\Omega$ and from the well-known theorem of Weierstrass we have that there exists the uniform convergence between its derivatives. From Representation Theorem one has  that   $\|f'_n-g'\|_{\infty}\to 0$. Thus from the previous reasoning  and the properties of the   infimum of a set   one obtains   the continuity of $\varphi_u$ by sequences.

On the other hand, using  \eqref{conj_Harm}, \eqref{ecua111}  and \eqref{ecua678}  one has that 
\begin{align*}
{ \bf P}_{\Omega}  \circ \varphi_u ( [f], ({\bf i},{\bf j}) )  = & 
{ \bf P }_{\Omega}   \left( [  D_1(f, u{\bf i}\bar u, u{\bf j}\bar u  ) ] ,   [ D_3(f, u{\bf i}\bar u, u{\bf j}\bar u)  ]    ,    (u{\bf i}\bar u, u{\bf j}\bar u )  \right)  \\
= & [f] .
\end{align*}
 for all $([f], ({\bf i}, {\bf j})) \in U\times T$.

Note that,  given  $u,v\in\mathbb S^3$,     $({\bf i}, {\bf j})\in T$     and    
$$[f] =[ P_{{u{\bf i}\bar u},{u {\bf j}\bar u }}(   \ (a  +c{u {\bf j}\bar u })+ {u{\bf i}\bar u}\int_{(x_0,y_0)} ^{(x ,y)}  -  ( a   + c  {u {\bf j}\bar u })_y  dx +   (a  + c   {u {\bf j}\bar u })_x   dy     \ ) ] \in {\bf B}_{\Omega} ,$$
 with $a,c\in CHarm^2(S_\Omega)$, one has that      
\begin{align*}    \varphi_u  ([  f],({\bf i},{\bf j})) =      (  \  [a] , [c]  ,    (v  (\bar v u ) {\bf i}  (\bar u v )\bar v , v(\bar v u) {\bf j}   (\bar u v)\bar v )   \  )  =   \varphi_v  ( [f], R_p({\bf i},{\bf j}) ),    
\end{align*}
where $p=\bar v u\in \mathbb S^3$. 
 
\end{proof}

\begin{rem}\label{REm01} In  $({\bf X}_{\Omega}, {\bf P}_{\Omega}, {\bf B}_{\Omega},  T  )$ the fibers of  $  {\bf B}_{\Omega}$ are
\begin{align*}{ ({\bf P}_{\Omega})}^{-1}([f]):=\{\left([ D_1(f,{\bf i},{\bf j} )]  ,  [D_3(f,{\bf i},{\bf j}) ]  ,   ({\bf i},{\bf j})  \right)  \ \mid \ ({\bf i},{\bf j})\in T \},\end{align*} where $ [f] \in {\bf B} $.
\end{rem}

 \begin{prop}\label{section}  The operator ${\bf S}_{\Omega, {\bf i}, {\bf j}}$ is a section of  $({\bf X}_{\Omega}, {\bf P}_{\Omega}, {\bf B}_{\Omega},  T  )$ for all $({\bf i}, {\bf j})\in T$.
	\end{prop}
\begin{proof}  	From  
formulas \eqref{conj_Harm}, 
 \eqref{ecua111}  and  \eqref{ecua678} one has  that  
\begin{align*} 
{  {\bf P}_{\Omega}} \circ {\bf S}_{{\bf i}, {\bf j}}([f]):= \left( \ [   D_1(f,  {\bf i} ,  {\bf j})  ]  ,  [  D_3(f,  {\bf i} ,  {\bf j} ) ]   ,
(  {\bf i} , {\bf j} ) \ \right) =[f] , \end{align*}         
for  all $ [f] \in {\bf B}_{\Omega} $ and the   	continuity of ${\bf S}_{ {\bf i},{\bf j}}$ is proved in the same way as   the continuity of   $\varphi_u$  in  the previous proposition.
\end{proof}

\begin{cor}\label{someFibers} Given $f\in C\mathcal {SR}(\Omega)$.
\begin{enumerate}
\item For each  $({\bf i}, {\bf j})\in T$ there are  unique   $[a],[c]\in C\textrm{Harm} (S_\Omega)\diagup \sim$ such that 
 $ ([ a],[ c]  ,  ({\bf i},{\bf j})  ) \in {\bf P}^{-1}(  [f] )
$.
 \item     $Q_{{\bf i}, {\bf j}}(f)\in \textrm{Hol} (\Omega_{\bf i})$ if and only if there exists $a \in C\textrm{Harm}(S_\Omega)$ such that 
 $( [ a] , [0] , ({\bf i},{\bf j}) ) \in {\bf P}^{-1}( [ f ])
$.  
\item  $f$ is an intrinsic slice regular functions on $\Omega$, see   \cite{CGS3,GS},
 if and only if  there exists $a \in C\textrm{Harm}(S_\Omega)$ such that  
 $ ( [a]  , 0  ,   ({\bf i},{\bf j}) ) \in {\bf P}^{-1}(  [f] )
$ for all  $({\bf i}, {\bf j})\in T$. 
\end{enumerate}
\end{cor}
\begin{proof} These facts follow  from the properties of  $P_{{\bf i}, {\bf j}}$ and $Q_{{\bf i}, {\bf j}}$.\end{proof}

\begin{rem}
It is important to comment that Propositions 
\ref{FiberBundle} and \ref{section} and Corollary \ref{someFibers} are deeply analogous to Propositions 3.6 and 3.8 and Corollary 3.9 given in \cite{O1}, respectively. But the sentences presented in this paper simplify the representation of each element of the   total and the base  spaces using the formula \eqref{conj_Harm}.   
 and  to achieve this goal   more requirements were established in our functions such as   two  equivalence relations
 and some conditions  
in the partial derivatives.
\end{rem}

The following operations are necessary to   show some algebraic properties of ${\bf P}_{\Omega}$.

\begin{defn}\label{defOper}Given 
$( [a] , [c] ,   ({\bf i},{\bf j})  ) , 
(  [p], [r]   ,   ({\bf i},{\bf j}) )
 \in {\bf X}$. Define 
 \begin{align*}  
 ( [a] , [c] ,   ({\bf i},{\bf j})  ) 
+ (  [p], [r]   ,   ({\bf i},{\bf j}) )
 := & ( [a+p] , [c+r] ,   ({\bf i},{\bf j})  ) , 
\\
 \mathcal D
\left(   [a] , [c] ,   ({\bf i},{\bf j})     \right)   
:= & ( [a_x] , [c_x] ,   ({\bf i},{\bf j})  ),
\\
 \mathcal R_u
\left( ( [a] , [c] ,   ({\bf i},{\bf j})  )  \right) 
:= &   ( [a] , [c] ,   (u{\bf i}\bar u, u{\bf j}\bar u) )
\end{align*}
where  $u\in \mathbb S^3$. For $f,g\in C\mathcal{SR}(\Omega)$ denote $[f]+[g]=[f+g]$ and $[f]' = [f']$.  \end{defn}

\begin{prop}   Given $({\bf i}, {\bf j})\in T$ and  
  $A,B \in {\bf X}_{\Omega}  $. One obtains that  
  \begin{align*}
	 {\bf P}_{\Omega} (A+B) =&   {\bf P}_{\Omega}(A) + {\bf P}_{\Omega} (B) ,\\ 
 ({\bf P}_{\Omega} (A))'= & {\bf P}_{\Omega}(\mathcal D (A)),\\
 \mathcal R_u(A) = & P_{u{\bf i}\bar u, u{\bf j}\bar u} [ \ u Q_{{\bf i}, {\bf j}} [{{\bf P}}_{\Omega} (A) ] \bar u \ ].
\end{align*}
\end{prop}
\begin{proof} These facts are consequences  of the previous definition. 
\end{proof}

\begin{rem} The operations  in Definition \ref{defOper} can be   represented by pullbacks in a similar way to   Remark 3.12 of \cite{O1}. Analogously to Proposition 3.14 given in  \cite{O1}, one can see that  the isomorphism of two quaternionic right linear space of slice regular functions associated to ${\bf i}$-simply connected bounded axially symmetric s-domain that  are 
${\bf i}$-conformally equivalents becomes an 
isomorphisms  of the  fiber bundles given in Proposition \ref{FiberBundle}.
 \end{rem}

On the other hand, we shall see a version of the Schwarz's Formula for quaternionic slice regular functions to get an idea of how to define a fiber bundle associated to the slice regular functions using \eqref{SchwarzForm}.

\begin{prop}\label{H1}Given    $f\in \mathcal{SR}_c(\mathbb B^4(0,\rho))$ and ${\bf i}, {\bf j} \in \mathbb S^2$  
 there exist $a,c\in  C{Harm}({\mathbb B^2(0,\rho)}) $ and $\lambda_1, \lambda_2\in\mathbb R$
such that 
\begin{align} \label{projection} 
f(q)= &  \frac{1}{2\pi}  \int_0^{2\pi}   (\rho e^{{\bf i}t} + q )*(\rho e^{{\bf i}t} - q)^{-*} 
				  ( \ a(\rho \cos t, \rho \sin t ) + c(\rho \cos t, \rho \sin t)  {\bf j}  \ ) dt \nonumber \\ 
			&		+ \lambda_1{\bf i} + \lambda_2{\bf ij}, 
								\end{align}
						 for all 	$ q\in \mathbb B^4(0,\rho)$ 	or equivalently
\begin{align*}  
 f(q)    = &
 	    u_{0,a,c} + \sum_{n=1}^{\infty}q^n 
u_{n,a,c}		, \quad \forall q\in\mathbb B^4(0,\rho),
				\end{align*}
 where the quaternionic coefficients are given by  
\begin{align*} 
 u_{0,a,c} &=   \frac{1}{2\pi}  \int_0^{2\pi} ( \ a(\rho \cos t, \rho \sin t ) + c(\rho \cos t, \rho \sin t)  {\bf j}  \ ) dt,	
 \\
u_{n,a,c}	&	=		\frac{1}{\pi} \int_0^{2\pi} (\rho e^{{\bf i} t})^{-n}( \ a(\rho \cos t, \rho \sin t ) + c(\rho \cos t, \rho \sin t)  {\bf j}  \ ) dt ,
\end{align*} 
for all $n\in \mathbb N$.
\end{prop}
\begin{proof}
Note that $Q_{{\bf i}, {\bf j}}[f]= f_1+f_2{\bf j}$,  where  $f_1,f_2\in \textrm{Hol}(\mathbb B^4(0,\rho)_{\bf i})$ and $f_1,f_2 \in C(\overline{\mathbb B^4(0,\rho)_{\bf i}}, \mathbb C)$. Denote   $a=\textrm{Re} (f_1)$ and 
$c=\textrm{Re} (f_2)$. Thus   from \eqref{SchwarzForm} there exist  $ \lambda_1, \lambda_2 \in \mathbb R$ such that 
\begin{align*}Q_{{\bf i}, {\bf j}} [f](z) = &  f_1(z) + f_2(z) {\bf j} \\
= & \frac{1}{2\pi}  \int_0^{2\pi} \left( a(\rho e^{it})\frac{\rho e^{it} + z }{\rho e^{it} - z} +  
c(\rho e^{it})\frac{\rho e^{it} + z }{\rho e^{it} - z} {\bf j}
 \right)  dt + i \lambda  \lambda_1 {\bf i} + \lambda_2 {\bf i j}
\end{align*}
and applying operator $P_{{\bf i}, {\bf j}}$ one get that 
 \begin{align*}
 f(q) = & \frac{1}{2\pi}  \int_0^{2\pi}P_{{\bf i},{\bf j}}( \ \frac{\rho e^{{\bf i}t} + z }{\rho e^{{\bf i}t} - z} \ ) (q)	( \ a(\rho \cos t, \rho \sin t ) + c(\rho \cos t, \rho \sin t)  {\bf j}  \ )   dt 
		 \nonumber\\ 
	    &\hspace{1cm}+   \lambda_1 {\bf i} + \lambda_2 {\bf i j}, \quad \forall q\in \mathbb B^4(0,\rho).
			\end{align*}
Identity  \eqref{projection} follows from   
 $$P_{{\bf i},{\bf j}}( \  \frac{\rho e^{{\bf i}t} + z }{\rho e^{{\bf i}t} - z} \ ) (q)= (\rho e^{{\bf i}t} + q )*(\rho e^{{\bf i}t} - q)^{-*}  , \quad \forall q\in \mathbb B^4(0,\rho).$$
On the other hand, for $|z|<\rho$ we see that 
\begin{align*}
\frac{\rho e^{{\bf i}t} + z }{\rho e^{{\bf i}t} - z} =  
 \left( 1 + \dfrac{z}{\rho e^{{\bf i}t}} \right)  \left(  \dfrac{1}{ 1 - \dfrac{ z}{ \rho e^{{\bf i}t} }  } \right)  = 1 +     \sum_{n=0}^{\infty}   z^n   \dfrac{ 2}{ \rho^n e^{n{\bf i}t} }  .  
\end{align*}
Therefore 
$$P_{{\bf i},{\bf j}}( \  \frac{\rho e^{{\bf i}t} + z }{\rho e^{{\bf i}t} - z} \ ) (q) = 1 +     \sum_{n=0}^{\infty}  q^n    2 (\rho e^{ {\bf i}t} )^{-n}   , \quad \forall q\in \mathbb B^4(0,\rho).$$
The uniform convergence of previous series 
 implies that
 \begin{align*}
 f(q) = & \frac{1}{2\pi}  \int_0^{2\pi} 	( \ a(\rho \cos t, \rho \sin t ) + c(\rho \cos t, \rho \sin t)  {\bf j}  \ )   dt \\
         &  
				+     \sum_{n=0}^{\infty}  q^n  
				\frac{1}{\pi}  \int_0^{\pi}(\rho e^{ {\bf i}t} )^{-n} 	( \ a(\rho \cos t, \rho \sin t ) + c(\rho \cos t, \rho \sin t)  {\bf j}  \ )   dt,
					\end{align*}
for all $q\in \mathbb B^4(0,\rho)$

\end{proof}

\begin{defn} \label{SeconFiberB}Given   $\mathbb B^2(0,\rho)=\{ (x,y) \in \mathbb R^2 \ \mid \ \|(x,y)\|_{\mathbb R^2} < \rho \}$ we shall consider     $  
C\textrm{Harm}( {\mathbb B^2(0,\rho)}) $,  $ C\mathcal{SR} (\mathbb B^4(0,\rho)) $ and given   
 $ a,c\in \textrm{Harm}_c (\mathbb B^2(0,\rho))$ and  $({\bf i},{\bf j})\in T$ define  
\begin{align*} {  { P}}_{\rho}  \left(   a, c , ({\bf i},{\bf j})  \right)    = & 
		       \frac{1}{2\pi}  \int_0^{2\pi}P_{{\bf i},{\bf j}}( \ \frac{\rho e^{{\bf i}t} + z }{\rho e^{{\bf i}t} - z} \ ) 
		( \ a(\rho \cos t, \rho \sin t ) + c(\rho \cos t, \rho \sin t)  {\bf j}  \ ) dt  , \end{align*}
	and for all   $({\bf i}, {\bf j})\in T $. Also denote     
$$\displaystyle 
{  \psi}_{\rho, u }( f, ({\bf i}, {\bf j}) )= \left( \      D_1(f,  u{\bf i}\bar u ,  u{\bf j}\bar u)     ,      D_3(f,  u{\bf i}\bar u ,  u{\bf j}\bar  u )     ,
(  u{\bf i} \bar  u, u{\bf j}\bar u ) \right)  ,       
$$   
for  all $f\in C\mathcal{SR} (\mathbb B^4(0,\rho)) $, $({\bf i},{\bf j}\in T)$ and all in $u\in \mathbb S^2$.

\end{defn}

\begin{rem}\label{Rem12}

One can  define suitable norms  and  equivalence relations in  
 $C\textrm{Harm}( {\mathbb B^2(0,\rho)}) $ and in  $ C\mathcal{SR} (\mathbb B^4(0,\rho)) $, 
  in a similar way to Definition \ref{DefFiberBundle},   to
	to find a fiber bundle  $({\bf X}_{\rho},  {\bf P}_{\rho}  ,{\bf B}_{\rho}, T  )$. Note that,  one has  a total space $X_{\rho} $  from   $C\textrm{Harm}( {\mathbb B^2(0,\rho)}) $  and a base space ${\bf B}_{\rho}$ from $ C\mathcal{SR} (\mathbb B^4(0,\rho)) $. Moreover,  the operator ${ P}_{\rho} $ allows us to find a bundle projection ${ \bf P}_{\rho} $  while from    ${  \psi}_{\rho, u } $ one  has an idea to define   the trivializations and the sections. Clearly, the structure group is $K=\{R_u  \mid u\in \mathbb S^3 \}$ and the  proofs of these facts  shall be similar to proof of   Proposition  \ref{FiberBundle}.
In addition, some properties of $({\bf X}_{\rho}, {\bf P}_{\rho}, {\bf B}_{\rho},  T  )$ such the representations of the sections, the  algebraic properties of ${\bf P}_{\rho}$, the interpretations of the operations in the base space   
 in terms of pullback bundles and so on  are obtained from a   similar way to those   presented  in Proposition 3.11, Remark 3.12 and Proposition 3.14 given  in \cite{O1}. That’s way we do not write more details of these properties.
\end{rem}

\begin{rem} The details of the fiber bundle $({\bf X}_{\rho}, {\bf P}_{\rho}, {\bf B}_{\rho},  T  )$  are omitted since 
   given ${\bf i}\in \mathbb S^2$ then a  
$ {\bf i}$-simply connected     bounded axially symmetric s-domain  $\Omega\subset\mathbb H$ and  $\mathbb B^4(0,\rho)$ are  ${\bf i}$-conformally equivalents  as a consequence from the Riemann's conformal mapping Theorem applied in the slices $\Omega_{\bf i}$ and $\mathbb B^4(0,\rho)_{\bf i}$ and  from  Proposition 3.14 given  \cite{O1} one obtains  an   isomorphism between the fiber bundles
 $({\bf X}_{\Omega}, {\bf P}_{\Omega}, {\bf B}_{\Omega},  T  )$ and $ (X_{\mathbb B^4(0,\rho)}, {\bf P}_{\mathbb B^4(0,\rho)}, B_{\mathbb B^4(0,\rho)}, T)$.

 \end{rem}

\subsection{Fiber bundles and the  zero sets of some slice regular functions}\label{ZeroSets}

This subsection shall be generalize the one to one relation  $Z_f \to f$, for all complex monic polymonial $f$, to some quaternionic slice regular polynomials using the theory of fiber bundles.

\begin{defn} 
The set  $\mathcal{SRB}(\mathbb B^4)$   consists of
quaternionic slice regular functions on $\mathbb B^4$ such that   
$$f(q)=\sum_{n=0}^{\infty} q^n r_n, \quad \forall q\in \mathbb B^4,$$
where the set $\{r_n \ \mid n\in \mathbb N \cup \{0\} \} \subset\mathbb H$ satisfy that     
 $\{{\bf r }_n \ \mid n\in \mathbb N \cup \{0\} \}$ contains a basis of $\mathbb R^3$ and  
     define  $\mathcal{PSRB}(\mathbb H):= \mathcal{PSR }(\mathbb H) \cap  \mathcal{SRB}(\mathbb B^4)$.

Given a family of sets $\Lambda = \{ A_{\bf i} \subset {\mathbb C}({\bf i}) \   \mid \  {\bf i}\in \mathbb S^2 \}$. The slice kull generated by $\Lambda$ is defined by 
$$   \textrm{SKull}\{ \Lambda\} := \bigcup_{{\bf i}\in \mathbb S^2} \textrm{Kull}(A_{\bf i}) ,  $$ 
where $\textrm{Kull}(A_{\bf i})$  is the kull generated by $A_{\bf i}$ in $\mathbb C({ \bf i})$.
 This concept can be justified by   the slice topology studied in  \cite{IXT1} that  extends  the concept of slice regularity 
 from a convenient topology in $\mathbb H$. In our case, we only are going use $  \textrm{SKull}\{ \Lambda\}$
  to explain a consequence of the    
	  Gauss-Lucas theorem in the elements of the base   and     total spaces of two fiber bundles induced by some slice regular polynomials.

\end{defn}

\begin{prop}\label{propol} Some properties of $\mathcal{SRB}(\mathbb B^4)$.
\begin{enumerate}

\item Given $f,g \in \mathcal{SRB}(\mathbb B^4)$. If there exist  ${\bf i}, {\bf i'}\in \mathbb S^2 $ different unit vectors, {}
 $c\in \mathbb C({\bf i})$ and $c'\in \mathbb C({\bf i'})$ such that  
 $$ f= g \bullet_{{\bf i}, {\bf j} } (1+ c{\bf j})   
\quad \textrm{and} \quad  
f= g \bullet_{{\bf i'}, {\bf j'} } (1+ c'{\bf j'}) ,$$
then $f=g$ on $\mathbb B^4$. 

\item If  $f,g \in \mathcal{PSRB}(\mathbb H)$ and  there exists ${\bf i}, {\bf i'}\in \mathbb S^2 $ such that 
 $Z_{f\pm {\bf i}f{\bf i}} \cap \mathbb C({\bf i}) = Z_{g\pm {\bf i}g{\bf i}} \cap \mathbb C({\bf i})$ and  $Z_{f\pm {\bf i'}f{\bf i'}} \cap \mathbb C({\bf i'}) = Z_{f\pm {\bf i'}g{\bf i'}} \cap \mathbb C({\bf i' })$, then $f=g$.

\end{enumerate}

\end{prop}

\begin{proof}
\begin{enumerate}

\item If $  f= g \bullet_{{\bf i}, {\bf j} } (1+ c{\bf j}) $ and 
$ f= g \bullet_{{\bf i'}, {\bf j'} } (1+ c'{\bf j'}) $. Therefore, from definition \eqref{iproducto} and identities \eqref{identity1} one has that 
 \begin{align*} f + {\bf i }f {\bf i } =  g + {\bf i }g {\bf i }   ,  & \quad 
 f - {\bf i }f {\bf i } =  cg  -{\bf i } c g {\bf i }, \quad     \textrm{  on } \  \  \mathbb C({\bf i}). \\ 
  f + {\bf i' }f {\bf i' } =  g + {\bf i' }g {\bf i' } , & \quad  f - {\bf i' }f {\bf i' } = c' g  -{\bf i' } c'g   {\bf i' } ,   \quad     \textrm{  on }   \  \   \mathbb C({\bf i'}). 
	\end{align*}
Then 
$ 2f  = (1+ c)g +  {\bf i } (1- c) g {\bf i }   $   on $\mathbb C({\bf i})$
  and 
$ 2f  = (1+ c')g +  {\bf i '} (1- c') g {\bf i '}   $   on $\mathbb C({\bf i'})$. 
 Thus denoting  $f(q)=\sum_{n=0}^{\infty} q^n a_n
 $ and $ g(q)=\sum_{n=0}^{\infty} q^n b_n
  $ for all $q\in \mathbb B^4$ one get that 
$$\sum_{n=0}^{\infty} z^n 2a_n    =   \sum_{n=0}^{\infty} z^n  (1+ c) b_n  +  \sum_{n=0}^{\infty} z^n   {\bf i } (1- c) b_n {\bf i }   
  $$  for all $z \in \mathbb C({\bf i})$ and 
$$\sum_{n=0}^{\infty} w^n 2a_n    =   \sum_{n=0}^{\infty} w^n  (1+ c') b_n  +  \sum_{n=0}^{\infty} w^n   {\bf i' } (1- c') b_n {\bf i '}   
  $$  for all $w \in \mathbb C({\bf i'})$.  Representation theorem  allows to see 
	that  
$$\sum_{n=0}^{\infty} q^n 2a_n    =   \sum_{n=0}^{\infty} q^n  ((1+ c) b_n + {\bf i } (1- c') b_n {\bf i } )   
  $$  and 
$$\sum_{n=0}^{\infty} q^n 2a_n    =   \sum_{n=0}^{\infty} q^n  ((1+ c') b_n + {\bf i '} (1- c') b_n {\bf i '} )   
  $$  for all $q \in \mathbb B^4$. Therefore 
	 $$ (1+ c) {\bf b}_n + {\bf i } (1- c) {\bf b}_n {\bf i }  =  (1+ c') {\bf b}_n + {\bf i '} (1- c') {\bf b}_n {\bf i '}       $$  for all $n\in \mathbb N$ and as a  basis of $\mathbb R^3$ is contained  in $\{ {\bf b}_n \ \mid \ n\in\mathbb N\cup\{0\} \}$   then 
	 $$ (1+ c) {\bf x} + {\bf i } (1- c) {\bf x} {\bf i }  =  (1+ c') {\bf x} + {\bf i '} (1- c') {\bf x} {\bf i '}       $$
 for all ${\bf x} \in \mathbb R^3$ and choosing some  vectors  ${\bf x}$   one has $c=c'=1$. Thus 
$f= g \bullet_{ {\bf i} , {\bf j}  } ( 1+ {\bf j})$, i.e.,  $f=g$.
  
\item Given  $f=f_1+ f_2{\bf j },g=g_1+ g_2{\bf j } \in \mathcal{PSRB}(\mathbb H)$, where $f_1,f_2,g_1,g_2\in \textrm{Hol}(\mathbb C({\bf i})) $ are complex polynomials. Note that  $f_1$ and $g_1$  are monic polynomials. Due to     
 $Z_f \cap \mathbb C({\bf i}) = Z_g \cap \mathbb C({\bf i})$ and 
$Z_f \cap \mathbb C({\bf i'}) = Z_g \cap \mathbb C({\bf i'})$  there exist  $c\in \mathbb C({\bf i})$ and 
$c'\in \mathbb C({\bf i'})$
 such that 
 $$ f= g \bullet_{{\bf i}, {\bf j} } (1+ c{\bf j})   
\quad \textrm{and} \quad  
f= g \bullet_{{\bf i'}, {\bf j'} } (1+ c{\bf j'}) .$$
From the  previous fact   $f=g$. 
\end{enumerate}

\end{proof}

\begin{rem} 
To show the importance of the function set $ \mathcal{SRB}(\mathbb B^4) $  in the previous proposition, let's look at   the following:
\begin{enumerate}

\item If $f$ is an intrinsic slice regular function on $\mathbb B^4(0,1)$, see \cite{CGS3}, there exists a sequence of real numbers $(\lambda_n)_{n\in \mathbb N\cup \{0\}}$ such that $\displaystyle f(q)=\sum_{n=0}^\infty q^n \lambda_n$. Thus  
$f = f \bullet_{{\bf i}, {\bf j} } (1+ c{\bf j}) $ for all ${\bf i}, {\bf j}\in \mathbb S^2$, orthogonal to each other, and all $c\in \mathbb C({\bf i})$.

\item If ${\bf i}, {\bf j}, {\bf k}, {\bf l} \in \mathbb S^2$ with  ${\bf i}$ and ${\bf j}$,   orthogonal to each other, and the same for the pair of vectors ${\bf k} $ and ${\bf l}$ such that ${\bf j}= {\bf k l}$. Then the functions 
 $f(q)=q (1+ {\bf j})$   and $g(q)=q (1-{\bf j}) $ are different and meet that  
$f = g \bullet_{{{\bf i}, {\bf j} }} (1-{\bf j})$ and  $f = g \bullet_{{{\bf k}, {\bf l} }} (1-{\bf l})$.

\end{enumerate}

	On the other hand, given ${\bf i}, {\bf j} \in \mathbb S^2$,   orthogonal to each other, the slice regular polynomials $f(q)=(q^2-1)  + (q-1){\bf j}$ and $g(q)=(q^2-1) 
	+ 7(q-1) {\bf j}$  
				satisfy that $f= g\bullet_{{{\bf i}, {\bf j} }} (1+ 7{\bf j})$. Then  
	$Z_{f\pm {\bf i}f{\bf i}} \cap \mathbb C({\bf i}) = Z_{g\pm {\bf i}g{\bf i}} \cap \mathbb C({\bf i})$.
	But $f\neq g$. That's why the identities   $Z_{f\pm {\bf i'}f{\bf i'}} \cap \mathbb C({\bf i'}) = Z_{g\pm {\bf i'}g{\bf i'}} \cap \mathbb C({\bf i'})$  are necessary in the    previous proposition to  show  that  $f=g$.

\end{rem}

\begin{prop}   \label{ZeroFiber}Denote
\begin{align*} X= &  \{  \left( Z_{f-{\bf i} f{\bf i}}  \cap \mathbb C({\bf i})  \right) \times 
\left( Z_{f+{\bf i} f{\bf i}}  \cap \mathbb C({\bf i})  \right)   \times 
\left( Z_{f-{\bf j} f{\bf j}}  \cap \mathbb C({\bf j})  \right) \times 
\left( Z_{f+{\bf j} f{\bf j}}  \cap \mathbb C({\bf j})  \right)  \\  
  & \hspace{1cm}\  \mid \ f\in \mathcal{PSRB}(\mathbb H) , \  ({\bf i}, 
{\bf j}) \in T \}    
 \end{align*}
and 
${\bf B}= \mathcal{PSRB}(\mathbb H)$ both sets equipped with the discrete topology. By  
$ {\bf P}: {\bf X} \to {\bf B} $   
 we mean the operator
$${\bf P}(   \ \left( Z_{f-{\bf i} f{\bf i}}  \cap \mathbb C({\bf i})  \right) \times 
\left( Z_{f+{\bf i} f{\bf i}}  \cap \mathbb C({\bf i})  \right)   \times 
\left( Z_{f-{\bf j} f{\bf j}}  \cap \mathbb C({\bf j})  \right) \times 
\left( Z_{f+{\bf j} f{\bf j}}  \cap \mathbb C({\bf j})  \right)  \  ) = f.$$
Also given ${u}\in \mathbb S^3$ and  let $U$ be   a neighborhood in $\mathcal{PSRB}(\mathbb H)$ denote 
$\phi_u : U\times T$   given  as follows: 
\begin{align*} 
 \phi_u (f, ({\bf i },{\bf j}))  
= &  
\left( Z_{f-u{\bf i}\bar u fu{\bf i}\bar u}  \cap \mathbb C(u{\bf i}\bar u)  \right) \times 
\left( Z_{f+u{\bf i}\bar u fu{\bf i}\bar u}  \cap \mathbb C(u{\bf i}\bar u)  \right)   
\\
 & \times 
\left( Z_{f-u{\bf j}\bar u fu{\bf j}\bar u}  \cap \mathbb C(u{\bf j}\bar u)  \right) \times 
\left( Z_{f+u{\bf j}\bar u fu{\bf j}\bar u}  \cap \mathbb C(u{\bf j}\bar u)  \right)
. \end{align*}
 
Given $({\bf i },{\bf j})\in T$ define  
\begin{align*}  S_{{\bf i },{\bf j}}(f)
 = & 
\left( Z_{f-{\bf i} f{\bf i}}  \cap \mathbb C({\bf i})  \right) \times 
\left( Z_{f+{\bf i} f{\bf i}}  \cap \mathbb C({\bf i})  \right)  \\
 & \times 
\left( Z_{f-{\bf j} f{\bf j}}  \cap \mathbb C({\bf j})  \right) \times 
\left( Z_{f+{\bf j} f{\bf j}}  \cap \mathbb C({\bf j})  \right).
\end{align*} 
    Then  $ ({\bf X}, {\bf P} ,{\bf B}, T  )$  is a fiber bundle where the structure group is $K=\{R_u  \mid u\in \mathbb S^3 \}$. The  family of trivializations      is  $\Phi =\{ \varphi_u    \mid  u\in\mathbb S^3\}$ and a family of sections is $\{S_{{\bf i },{\bf j}} \  \mid \  ({\bf i },{\bf j})\in T\}. $ 
\end{prop}
\begin{proof} 
The fact 2 of the previous proposition establishes the well-definition of the bundle projection  ${\bf P} $ and also helps us to see that 
${\bf P}\circ  \phi_u (f, ({\bf i },{\bf j}))  = f$ for all $f\in  \mathcal {PSRB}(\mathbb H)$ and  all $({\bf i}, { \bf j})\in T$.
 Fixing $({\bf i}, { \bf j})\in T$ and doing $u=1$ one sees the section   $ S_{{\bf i },{\bf j}}(f)= \phi_1 (f, ({\bf i },{\bf j}))  $
for all $f\in  \mathcal {PSRB}(\mathbb H)$. 
\end{proof}

\begin{rem}
Note that ${\bf X}$ and ${\bf B}$ are equipped with the discrete topology since the important thing  to describe the relationship between  some    slice regular polynomials 
	with their zero sets.

The   mapping  from the finite subsets of $\mathbb C$ to the set  of    
 the  monic complex polynomial $Z_f\mapsto f$ 
  is a bijective mapping.  In this sense Proposition \ref{ZeroFiber}  presents  an extension of this   phenomena  to  $\mathcal{SRB}(\mathbb H)$   explained in terms of the fiber bundle theory and complements the results of the zero sets of slice regular functions presented in \cite{BW,  GSt}
\end{rem}

\begin{defn}We shall establish two fiber bundles: \begin{enumerate}
\item Denote $ {\bf P}\mathcal{SRB}(\mathbb H) = \{f \in \mathcal{PSR}(\mathbb H)  \  \mid \   f' \in \mathcal{PSRB}(\mathbb H)  \} $,
\begin{align*} X_1= &  \{  \left( Z_{ f'  -{\bf  i}  f' {\bf  i}}  \cap \mathbb C({\bf  i})  \right) \times 
\left( Z_{ f' +{\bf  i}  f' {\bf  i}}  \cap \mathbb C({\bf  i})  \right)   \times 
\left( Z_{ f' -{\bf   j}  f' {\bf  j}}  \cap \mathbb C({\bf  j})  \right) \times 
\left( Z_{ f' +{\bf  j}  f' {\bf  j}}  \cap \mathbb C({\bf  j})  \right)  \\  
  & \hspace{1cm}\  \mid \  f  \in {\bf P}\mathcal{PSRB}(\mathbb H) , \  ({\bf  i}, 
{\bf  j}) \in T \} ,   
 \end{align*}
${\bf B}_1=\{ f' \ \mid \ f\in {\bf P}\mathcal{SRB}(\mathbb H)\}$, both equipped with the discrete topology, and define   the mapping  
 $ {\bf P}_1: {\bf X_1} \to {\bf B_1} $   
 as follows 
\begin{align*}&{\bf P}_1(   \ \left( Z_{f'-{\bf i} f'{\bf i}}  \cap \mathbb C({\bf i})  \right) \times 
\left( Z_{f'+{\bf i} f'{\bf i}}  \cap \mathbb C({\bf i})  \right)   \times 
\left( Z_{f'-{\bf j} f'{\bf j}}  \cap \mathbb C({\bf j})  \right) \times 
\left( Z_{f'+{\bf j} f'{\bf j}}  \cap \mathbb C({\bf j})  \right)  \  )  \\  
& = f'.\end{align*}
In addition, define
\begin{align*} 
 \varphi_u^1 (f', ({\bf i },{\bf j}))  
= &  
\left( Z_{f'-u{\bf i}\bar u f'u{\bf i}\bar u}  \cap \mathbb C(u{\bf i}\bar u)  \right) \times 
\left( Z_{f'+u{\bf i}\bar u f'u{\bf i}\bar u}  \cap \mathbb C(u{\bf i}\bar u)  \right)   
\\
 & \times 
\left( Z_{f'-u{\bf j}\bar u f'u{\bf j}\bar u}  \cap \mathbb C(u{\bf j}\bar u)  \right) \times 
\left( Z_{f'+u{\bf j}\bar u f'u{\bf j}\bar u}  \cap \mathbb C(u{\bf j}\bar u)  \right)
, \end{align*}
where   ${u}\in \mathbb S^3$ and   $U\subset P\mathcal{PSRB}(\mathbb H)$ is a neighborhood. Set   
	$K=\{R_u  \mid u\in \mathbb S^3 \}$ and     $\Phi_1 =\{ \varphi_u^1    \mid  u\in\mathbb S^3\}$.

\item Consider 
$${\bf X}_2 : = \{  \textrm{Kull}(Z_f \cap \mathbb {C}( {\bf i})) \times \textrm{Kull}( Z_f \cap \mathbb {C}( {\bf j})) \ \mid \ f \in {\bf P}\mathcal{PSRB}(\mathbb H)  , \   ({\bf i}, {\bf j}  )\in T \},$$
  ${\bf B}_2:=   \{  \textrm{SKull}(\Lambda_f)  \ \mid \  f \in {\bf P}\mathcal{PSRB}(\mathbb H)   \}$, 
	both equipped with the discrete topology, and define 
$ {\bf P}_2: {\bf X}_2 \to {\bf B}_2 $ as follows:  
\begin{align*} {\bf P}_2( \textrm{Kull}(Z_f \cap \mathbb {C}( {\bf i})) \times \textrm{Kull}( Z_f \cap \mathbb {C}( {\bf j})) )=&
\bigcup_{ {\bf v} \in \mathbb S^2}\textrm{Kull}(Z_f \cap \mathbb {C}( {\bf v})) \\
=  &
\textrm{SKull}(\Lambda_f).\end{align*}

Given a neighborhood $U\subset {\bf B}_2
$ define $\phi_u^2 : U\times T$ as follows  
\begin{align*}\varphi_u^2 (\textrm{SKull}(\Lambda_f), ({\bf i },{\bf j})) = &  \textrm{SKull}(\Lambda_f) \cap  \mathbb {C}( u{\bf i}\bar  u) \times \textrm{SKull}(\Lambda_f) \cap  \mathbb {C}( u{\bf j}\bar  u)\\
= &
\textrm{Kull}(Z_f \cap \mathbb {C}( u{\bf i}\bar u)) \times \textrm{Kull}( Z_f \cap \mathbb {C}(u {\bf j}\bar u))  .\end{align*}
Set 
$K=\{R_u  \mid u\in \mathbb S^3 \}$ and     $\Phi_2 =\{ \varphi_u^2    \mid  u\in\mathbb S^3\}$.

\end{enumerate}

\end{defn}

\begin{prop} $({\bf X}_1, {\bf P}_1 ,{\bf B}_1, T  )$   and      $  ({\bf X}_2, {\bf P}_2 ,{\bf B}_2, T  )$  are 
 fiber bundles and there exists a  morphism  $\Gamma : (X_1, {\bf P}_1, B_1, F_1) \to   (X_2, {\bf P}_2, B_2, F_2)$.
\end{prop}

\begin{proof}   $({\bf X}_1, {\bf P}_1 ,{\bf B}_1, T  )$ 
 is a fiber bundle as a consequence of the previous proposition and         $  ({\bf X}_2, {\bf P}_2 ,{\bf B}_2, T  )$  is a fiber bundle directly from its definition. The morphisms are the following:  
   \begin{align*}
	& \Gamma_1 [  \     
		\left( Z_{ f'  -{\bf  i}  f' {\bf  i}}  \cap \mathbb C({\bf  i})  \right) \times 
\left( Z_{ f' +{\bf  i}  f' {\bf  i}}  \cap \mathbb C({\bf  i})  \right)   \times 
\left( Z_{ f' -{\bf   j}  f' {\bf  j}}  \cap \mathbb C({\bf  j})  \right) \times 
\left( Z_{ f' +{\bf  j}  f' {\bf  j}}  \cap \mathbb C({\bf  j})  \right)
		\   ] \\
&=	\textrm{Kull}(Z_f \cap \mathbb {C}( {\bf i})) \times \textrm{Kull}( Z_f \cap \mathbb {C}( {\bf j})) 
	\end{align*}
	and   	
$\Gamma_2 [ f'  ] 
=	\textrm{SKull}(\Lambda_f)$,
	 	for all $f\in {\bf P}\mathcal{SRB}(\mathbb H)$ and  $({\bf i}, {\bf j})\in T$.
	\end{proof}

\begin{rem} 
Given $ f\in {\bf P}\mathcal{SRB}(\mathbb H)$  and  $({\bf i}, {\bf j})\in T$. Then 
  Gauss Lucas Theorem applied in the complex components of  
	$f\mid_{\mathbb C({\bf i})}$
	 gives us the following contention     
 $$ Z_{ f'  -{\bf  i}  f' {\bf  i}}  \cap \mathbb C({\bf  i})   \cap  
 Z_{ f' +{\bf  i}  f' {\bf  i}} =Z_{f'}\cap \mathbb C({\bf i})  \subset \textrm{Kull}(Z_f \cap \mathbb {C}( {\bf i}))  $$ 
and  
$$   Z_{ f'}  =  \bigcup_{{\bf i} \in \mathbb S^2}    Z_{  f'}\cap \mathbb {C}({ \bf i }) \subset   
\bigcup_{{\bf i} \in \mathbb S^2}  \textrm{Kull} ( \ Z_{  f}\cap \mathbb {C}({ \bf i })
 \ ) ,$$
i.e., $  Z_{ f'} \subset \textrm{SKull}\{ \Lambda_f \} $,  
where $\Lambda_f:= \{    Z_{  f}\cap \mathbb {C}({ \bf i })
  \ \mid \ { \bf i}\in\mathbb S^2  \}$.

Therefore,  Gauss-Lucas Theorem shows us another     relationship between   $A\in X_1$ and  $\Gamma_1(A)\in X_2$ and gives us another point of view   of the mapping  $ f' \mapsto 	\textrm{SKull}(\Lambda_f)$ 
	 	for all $f\in {\bf P}\mathcal{SRB}(\mathbb H)$.

It is important to comment that 
the facts presented in this subsection complement  the results presented in 
 \cite{V}  which shows a quaternionic version of the Gauss-Lucas Theorem for quaternionic slice regular polynomials.
	\end{rem}

  \end{document}